\documentclass[12pt, reqno]{amsart}
\usepackage{amsmath, amsthm, amscd, amsfonts, amssymb, graphicx, color}
\usepackage[bookmarksnumbered, colorlinks, plainpages]{hyperref}
\newtheorem{thm}{Theorem}[section]
 \newtheorem{cor}[thm]{Corollary}
 \newtheorem{lem}[thm]{Lemma}
 \newtheorem{prop}[thm]{Proposition}
 \theoremstyle{definition}
 \newtheorem{defn}[thm]{Definition}
 \theoremstyle{remark}
 
 \newtheorem*{ex}{Example}
 \numberwithin{equation}{section}
\begin{document}
\title[Liouville-type results for stationary maps]{Liouville-type results for stationary maps of a class of functional related to pullback metrics }
\author{Sa\"{\i}d Asserda}
\address{Ibn tofail university , Faculty of sciences, PO 242 Kenitra Morroco}
\email{asserda-said@univ-ibntofail.ac.ma}
\begin{abstract}
 We
 study a generalized functional related to the pullback metrics (3).
We derive the first variation formula  which yield stationary
maps. We introduce the stress-energy tensor which is naturally
linked to conservation law and yield monotonicity formula via the
coarea formula and comparison theorem in Riemannian geometry. A
version of this monotonicity inequalities  enables us to derive
some Liouville type results. Also we investigate the constant
Dirichlet boundary value problems and the generalized Chern type
results for tension field equation with respect to this
functional.
\end{abstract}
\maketitle
\section{Introduction}
Let  $u : (M^{m},g)\rightarrow (N^{n},h)$ be a smooth map between
Riemannian manifolds $(M^{m},g)$ and $(N^{n},h)$. Recently, Kawai
and Nakauchi [7] introduce a functional related to the pullback
metric $u^{*}h$ :
\begin{equation}\Phi(u)={1\over 4}\int_{M}\Vert u^{*}h\Vert^{2}dv_{g}\end{equation}
where $u^{*}h$ is the symetric $2$-tensor ( pullback metric )
defined by $$(u^{*}h)(X,Y)=<du(X),du(Y)>_{h}$$ for any vector
fields $X,Y$ on $M$ and $\Vert u^{*}h\Vert$ its norm $$\Vert
u^{*}h\Vert=\sqrt{\sum_{i;j=1}^{m}\left(\left<du(e_{i}),du(e_{j})\right>_{h}\right)^{2}}$$
 with respect to  a local orthonormal frame $(e_{1}\cdots,e_{m})$ on $(M,g)$. The map $u$ is stationary for $\Phi$ if it is a critical point of $\Phi(u)$ with respect to any compactly supported variation of $u$ and $u$ is  stationary stable  if the second variation for the functional $\Phi(u)$ is non-negative. When $M$ and $N$ are compact without boundary , the same authors show the non-existence of non-constant stable stationary map for $\Phi$ if $M$ (respectively $N$) is a standard sphere $\mathbb{S}^{m}$ ( respectively $\mathbb{S}^{n}$). Also they show that a stationary map of $\Phi$ is a constant map provided that $M$ is compact without boundary and $N$ is non-compact supporting a strictly $C^{2}$ convex function.  \\
\indent On the other hand, following Baird and Eells [2],  M.Ara
[1]  introduced the  $F$-harmonic maps, generalizing  harmonic
maps, and the stress $F$-energy tensor. Let $F :
[0,+\infty[\rightarrow [0,+\infty[$ be a $C^{2}$ function such
that $F(0)=0$ and $F^{'}(t)>0$ on $t\in]0,+\infty[$. A smooth map
$ u : M\rightarrow N$ is said to be an $F$-harmonic map if it is a
critical point of the following $F$-energy functional $E_{F}$
given by
\begin{equation}E_{F}(u)=\int_{M}F\left({\Vert du\Vert^{2}\over 2}\right)dv_{g},\end{equation} with respect to any compactly supported variation of $u$, where $\Vert du\Vert$ is the Hilbert-Schmid norm of the differential of $u$ :
$$\Vert du\Vert^{2}=trace_{g}u^{*}h=\sum_{i=1}^{m}<du(e_{i}),du(e_{j})>_{h}.$$
The stress $F$-energy tensor $S_{F}$ associated with $E_{F}$-energy  is given  at any map $u$ by  $$S_{F,u}(X,Y)=F\left({\Vert du\Vert^{2}\over 2}\right)<X,Y>_{g}-F^{'}\left({\Vert du\Vert^{2}\over 2}\right)<du(X),du(Y)>_{h}$$ for any vector fields $X,Y$ on $M$. Via the stress-energy tensor $S_{F}$ of $E_{F}$, monotonicity formula, Liouville-type results and the constant Dirichlet boundary-value problem were investigated recently by Dong and Wei  generalizing and refining the works of several authors ( see [5] and references therein) \\
\indent In this paper,  we generalize and unify the concept of
critical point of the functional $\Phi$. For this, we define the
functional $\Phi_{F}$ by
\begin{equation}\Phi_{F}(u)=\int_{M}F\left({\Vert
u^{*}h\Vert^{2}\over 4}\right)dv_{g}\end{equation} which is $\Phi$
if $F(t)=t$. We derive the first variation formula of $\Phi_{F}$
and we introduce the stress-energy tensor $S_{\Phi_{F}}$
associated to $\Phi_{F}$ which is naturally linked to conservation
law. The tensor $S_{\Phi_{F}}$ yield monotonicity formula via
coarea formula and comparison theorems in Riemannian geometry.
These monotonicity inequalities enable us to derive a large
classes of Liouville-type results for stationary maps for the
functional $\Phi_{F}$. As another consequence, we obtain the
unique constant solutions of constant Dirichlet boundary-value
problems on starlike domains for smooth maps satisfying a
conservation law. We also obtain  generalized Chern type results
for tension field equation with respect to the functional
$\Phi_{F}$. We mention that our results are extentions of results
of Nakauchi-Takenaka where they gave the first varaition foirmula,
the second variation formule,
the monotonicity formula and the Bochner type formula [10]. \\ \\
 \noindent The contents of this paper is as follows :\\ \\
 1- Introduction\\
 2- Functionals related to pullback metrics and conservation law\\
 3- Monotonicity formula and Liouvile-type results\\
 4- Constant Dirichlet boundary-value problems\\
 5- Generalized Chern type results
\section{Functionals related to pullback metrics and conservation law}
Let $F : [0,+\infty[\rightarrow[0,+\infty[$ be a $C^{2}$-function
such that $F(0)=0$ and $F^{'}>0$ on $]0,+\infty[$. let $ u:
M\rightarrow N$ be a smooth map from an $m$-dimensional Riemannian
manifold $(M,g)$ to a Riemannian manifolf $(N,h)$. We call $u$ a
stationary map for the functional
$$\Phi_{F}(u)=\int_{M}F\Bigl({\Vert u^{*}h\Vert^{2}\over 4}\Bigr)dV_{g}$$
if
$${d\over dt}\Phi_{F}(u_{t})|_{t=0}=0$$ for any compactly supported variation $u_{t} : M\rightarrow N ( -\epsilon<t<\epsilon)$ with $u_{0}=u$\\ \\
Let $\nabla$ and ${}^{N}\nabla$ always denote the Levi-civita
connetions of $M$ and $N$ respectively. Let $\tilde\nabla$ be the
induced connection on $u^{-1}TN$ defined by
${\tilde\nabla}_{X}W={}^{N}\nabla_{du(X)}W,$ where $X$ is a
tangent vector of $M$ and $W$ is a section of $u^{-1}TN$. We
choose a local orthonormal frame field $\{e_{i}\}_{i=1}^{m}$ on
$M$. We define an $u^{-1}TN$-valued $1$-forme $\sigma_{F,u}$ on
$M$ by \begin{equation}\sigma_{F,u}(.)=F^{'}\Bigl({\Vert
u^{*}h\Vert^{2}\over
4}\Bigr)\sum_{j=1}^{m}h(du(.),du(e_{j}))du(e_{j}).\end{equation}
When $F(t)=t,$ we have
$\sigma_{F,u}(.)=\sigma_{u}(.)=\sum_{j=1}^{m}h(du(.),du(e_{j}))du(e_{j})$,
as defined in [7], which give
$$\Vert u^{*}h\Vert^{2}=\sum_{i=1}^{m}<du(e_{i}),\sigma_{u}(e_{i})>_{h}$$ We define the tension field $\tau_{\Phi_{F}}(u)$  of $u$ with respect the functional $\Phi_{F}$ by
\begin{equation}
\tau_{\Phi_{F}}(u):=div_{g}\sigma_{F,u}\end{equation} where
$div_{g}\sigma_{F,u}$ denotes the divergence of $\sigma_{F,u}$ :
  \begin{eqnarray*}\tau_{\Phi_{F}}(u)&=&\sum_{i=1}^{m}\Bigl\{{\tilde\nabla}_{e_{i}}\Bigl(F^{'}\Bigl({\Vert u^{*}h\Vert^{2}\over 4}\Bigr)\sigma_{u}(e_{i})\Bigr)-F^{'}\Bigl({\Vert u^{*}h\Vert^{2}\over 4}\Bigr)\sigma_{u}(\nabla_{e_{i}}e_{i})\Bigr\}\\
&=&F^{'}\Bigl({\Vert u^{*}h\Vert^{2}\over
4}\Bigr)div_{g}\sigma_{u}+\sigma_{u}\Bigl\{\hbox{grad}\Bigl(F^{'}\Bigl({\Vert
u^{*}h\Vert^{2}\over 4}\Bigr)\Bigr)\Bigr\}
\end{eqnarray*}
\begin{lem}[First variation formula] Let $u : (M,g)\rightarrow (N,h)$ be a $C^{2}$ map. Then
  $${d\over dt}\Phi_{F}(u_{t})\big|_{t=0}=-\int_{M}<\tau_{\Phi_{F}}(u),du(X)>_{h}dV_{g},$$ where $X$ ( respectively $u_{t}$) is  any smooth vector field  with compact support on $M$ ( respectively any $C^{2}$ deformation of $u$).
  \end{lem}
  \begin{proof}Let $X$ be a vector field on $M$ with compact support. Let $U : ]-\epsilon,\epsilon[\times M\rightarrow N$ be any smooth deformation of $u$ such that
\begin{eqnarray*}
 U(0,x)&=&u(x)\\
 dU({\partial\over\partial t},X)\big|_{t=0}&=&du(X)
 \end{eqnarray*}
Put $u_{t}(x)=U(t,x)$ and $Y_{t}=du_{t}({\partial\over\partial
t},X)$ the variation field. Then
\begin{eqnarray*}
{\partial\over\partial t}F\left ({\Vert u_{t}^{*}h\Vert^{2}\over 4}\right )&=&{\partial\over\partial t}F\left ({\sum_{i,j=1}^{m}\left<du_{t}(e_{i}),du_{t}(e_{j})\right>^{2}_{h}\over 4}\right )\\
&=&F^{'}\left ({\Vert u_{t}^{*}h\Vert^{2}\over 4}\right ){1\over 4}{\partial\over\partial t}\left (\sum_{i,j=1}^{m}\left<du_{t}(e_{i}),du_{t}(e_{j})\right>_{h}\left<du_{t}(e_{i}),du_{t}(e_{j})\right>_{h}\right )\\
&=&F^{'}\left ({\Vert u_{t}^{*}h\Vert^{2}\over 4}\right )\sum_{i,j=1}^{m}\left<\nabla_{\partial\over\partial t}\left (du_{t}(e_{i})\right),du_{t}(e_{j})\right>_{h}\left<du_{t}(e_{i}),du_{t}(e_{j})\right>_{h}\\
&=&\sum_{i=1}^{m}\left< \sigma_{F,u_{t}}(e_{i}),{\tilde
\nabla}_{e_{i}}Y_{t}\right>_{h}
\end{eqnarray*}
Since $Y_{t}$ has compact support on $M$, using an integration by
parts we obtain
\begin{eqnarray*}
{d\over
dt}\Phi_{F}(u_{t})\big|_{t=0}&=&-\int_{M}\left<\sum_{i=1}^{m}{\tilde
\nabla}_{e_{i}}\left(
\sigma_{F,u}(e_{i})\right),Y_{0}\right>_{h}dv_{g}\\
&=&-\int_{M}\left<div_{g}\left(\sigma_{F,u}\right),du(X)\right>_{h}dv_{g}
\end{eqnarray*}
which is the first variation formula for $\Phi_{F}$.
\end{proof}
The first variation formula allows us to define the notion of
stationary maps for the functional $\Phi_{F}$.
\begin{defn}[Stationary map]
 A smooth map $u$ is called  stationary  for the functional $\Phi_{F}$ if it is a solution of the Euler-Lagrange equation
 $$
 div_{g}\sigma_{F,u}=0,
 $$
 equivalently
$$
F^{'}\Bigl({\Vert u^{*}h\Vert^{2}\over
4}\Bigr)div_{g}\sigma_{u}+\sigma_{u}\Bigl(\hbox{grad}\Bigl(F^{'}\Bigl({\Vert
u^{*}h\Vert^{2}\over 4}\Bigr)\Bigr)\Bigr\}=0$$
\end{defn}
\begin{ex}
\noindent 1- Every totally geodesic map $u$ , i.e $\nabla du=0$, is stationary for $\Phi_{F}$ \\
2- If $N=\mathbb{R}$ then $\Vert u^{*}h\Vert^{2}=\Vert du\Vert^{4}$. Hence $u$ is stationary for $\Phi_{F}$ if and only if $u$ is $G$-harmonique where $G(t)=F(t^{2})$.
\end{ex}
\indent Following Baird [3], for a smooth map $u :
(M,g)\rightarrow(N,h)$  we associate a symmetric $2$-tensor
$S_{\Phi_{F,u}}$ to the functional $\Phi_{F}$, called the
stress-energy tensor
\begin{equation}
S_{\Phi_{F,u}}(X,Y)=F\left({\Vert u^{*}h\Vert^{2}\over
4}\right)\left<X,Y\right>_{g}-F^{'}\left({\Vert
u^{*}h\Vert^{2}\over 4}\right)\left<\sigma_{u}(X),du(Y)\right>_{h}
\end{equation}
where $X,Y$ are vector fields on $M$.
\begin{prop}
 Let $u : (M,g)\rightarrow(N,h)$ a smooth map and let $S_{\Phi_{F,u}}$ be the associatd stress-energy tensor, then for all $x\in M$ and for each vector $X\in T_{x}M$, $$(\hbox{div} S_{\Phi_{F,u}})(X)=-\left<\hbox{div}_{g}\sigma_{F,u},du(X)\right>$$
where $$\sigma_{F,u}(X)=F^{'}\left({\Vert u^{*}h\Vert^{2}\over
4}\right)\sum_{i=1}^{m}\left<du(X),du(e_{i})\right>du(e_{i}).$$
\end{prop}
\begin{proof}
For any vector field $X$ on $M$ we have  \begin{eqnarray*}{1\over
4}\nabla_{X}\Vert
u^{*}h\Vert^{2}&=\sum_{i,j=1}^{m}\left<\left(\nabla_{X}du\right)(e_{i}),du(e_{j})\right>\left<du(e_{i}),du(e_{j})\right>
\\
&=\sum_{i=1}^{m}\left<\left(\nabla_{X}du\right)(e_{i}),\sigma_{u}(e_{i})\right>
\end{eqnarray*}
We compute
\begin{eqnarray*}
(div
S_{\Phi_{F,u}})(X)&=&\sum_{k=1}^{m}\left(\nabla_{e_{k}}\left(S_{\Phi_{F,u}}(e_{k},X)\right)-
S_{\Phi_{F,u}}(e_{k},\nabla_{e_{k}}X)-S_{\Phi_{F,u}}(\nabla_{e_{k}}e_{k},X)\right)\\
&=&\sum_{k=1}^{m}\Big\{\left(\nabla_{e_{k}}F({\Vert
u^{*}h\Vert^{2}\over 4})\right)<e_{k},X>+F({\Vert
u^{*}h\Vert^{2}\over 4})<e_{k},\nabla_{e_{k}}X>
-\\
&{}&\quad \nabla_{e_{k}}\left(<\sigma_{F,u}(e_{k}),du(X)>\right)-F({\Vert u^{*}h\Vert^{2}\over 4})<e_{k},\nabla_{e_{k}}X>\\
&{}&\quad +<\sigma_{F,u}(e_{k}),du(\nabla_{e_{k}}X)>\Big\}\\
&=&\nabla_{X}F({\Vert u^{*}h\Vert^{2}\over
4})-<div_{g}\sigma_{F,u},du(X)>-\\
&{}&\quad 
\sum_{k=1}^{m}\left<\sigma_{F,u}(e_{k}),\nabla_{e_{k}}(du(X))-du(\nabla_{e_{k}}X)\right>\\
&=&-<div_{g}\sigma_{F,u},du(X)>+\sum_{i=1}^{m}\left<(\nabla_{X}du)(e_{i}),\sigma_{F,u}(e_{i})\right>
\sum_{k=1}^{m}<(\nabla_{e_{k}}du)(X),\sigma_{F,u}(e_{k})>
\end{eqnarray*}
Since $(\nabla_{e_{k}}du)(X)=(\nabla_{X}du)(e_{k})$ we obtain
$$(div S_{\Phi_{F,U}})(X)=-<div_{g}\sigma_{F,u},du(X)>$$
\end{proof}
\begin{defn}
 A map $u : M\rightarrow N$ is said to satisfy an $\Phi_{F}$-conservation law if $S_{\Phi_{F,u}}$ is divergence free, i.e the $(0,1)$-type tensor field $\hbox{div}S_{\Phi_{F,u}}$ vanishes identically (div$S_{\Phi_{F,u}}=0$).
 \end{defn}
In particular if $u$ is a stationary map for $\Phi_{F}$ then
$S_{\Phi_{F,u}}$ is divergence free. In general the conserve will
not be true i.e $\hbox{div}S_{\Phi_{F,u}}=0$, then we cannot
conclude that $u$ is a stationary map for the functional
$\Phi_{F}$. However, in the spacial case when $u$ is a submersive
mapping, we do have the equivalence.
\begin{cor}
 \ If $ u :M\rightarrow N$ is a submersion almost everywhere, then $u$ is a stationary map for $\Phi_{F}$ if and only if $\hbox{div}S_{\Phi_{F,u}}=0$.
 \end{cor}
\section{Monotonicity formula and Liouville-type results}
In this section, we will establish the monotonicity formula for
the functional $\Phi_{F}$ on complete Riemannian manifolds with a
pole. We recall a pole is a point $x_{0}\in M$ such that the
exponential map $ \exp_{x_{0}} : T_{x_{0}}M\rightarrow M$ is a
global diffeomorphism. For a map : $u : (M,g)\rightarrow (N,h)$
satisfying an $\Phi_{F}$-conservation law , the stress-energy
$2$-tensor $S_{\Phi_{F,u}}$   yield monotonicity formula via corea
formula and comparison theorems in Riemannian geometry. These
monotonicity inequalities enable us to derive a large classes of
Liouville-type results for stationary maps for the functional
$\Phi_{F}$. We mention that Nakauchi and Takenaka gave the
monotonicity formula
for the functional $\Phi$ and our result is a generalization of their result to the $F$-functional  $\Phi_{F}$ [10].\\ For a vector field $X$ on $M$, we denote by $\theta_{X}$ its dual
$1$-form, i.e
$$\theta_{X}(Y)=<X,Y>_{g}$$
Let $\nabla\theta_{X}$ the $2$-tensor \begin{eqnarray*}(\nabla\theta_{X})(Y,Z)&=&(\nabla_{Y}\theta_{X}(Z)\\
&=&Y(\theta_{X}(Z))-\theta_{X}(\nabla_{Y}Z)\\
&=&<\nabla_{Y}X,Z>_{g}
\end{eqnarray*}
Following Baird [3], the contraction of  stress-energy tensor
$S_{\Phi_{F,u}}$ with  $X$ is given by
\begin{equation}\hbox{div}\left(i_{X}S_{\Phi_{F,u}}\right)=\left(\hbox{div}S_{\Phi_{F,u}}\right)(X)+
\left<S_{\Phi_{F,u}},\nabla\theta_{X}\right>\end{equation} where
$(i_{X}S_{\Phi_{F,u}})(Y):=S_{\Phi_{F,u}}(X,Y)$ and
$$
\left<S_{\Phi_{F,u}},\nabla\theta_{X}\right>=
\sum_{i,j=1}^{m}S_{\Phi_{F,u}}(e_{i},e_{j})\left<\nabla_{e_{i}}X,e_{j}\right>_{g}.$$
Let $D$ be any bounded domain of $M$ with $C^{1}$-boundary. We
integrate the formula (7), by Stokes's theorem, we obtain the
basis of monotonicity formula
\begin{equation}\int_{\partial D}S_{\Phi_{F,u}}(X,\nu)ds_{g}=\int_{D}\left\{\left(\hbox{div}S_{\Phi_{F,u}}\right)(X)+
\left<S_{\Phi_{F,u}},\nabla\theta_{X}\right>\right\}dv_{g}\end{equation}
where $\nu$ is unit outward normal vector field along $\partial D$
with $(m-1)$-dimensional volume element $ds_{g}$. In particular,
if $u$ satisfies an $\Phi_{F}$-conservation law, for almost
$R_{2}>R_{1}\geq 0$ we have
\begin{equation}\int_{\partial B(R_{2})}S_{\Phi_{F,u}}(X,\nu)ds_{g}-\int_{\partial B(R_{1})}S_{\Phi_{F,u}}(X,\nu)ds_{g}=\int_{B(R_{2})\setminus B(R_{1})}
\left<S_{\Phi_{F,u}},\nabla\theta_{X}\right>dv_{g}
\end{equation}
Following Kassi[8], we introduce the following
\begin{defn}
 \ The upper (lower)  $F$-degree $d_{F}$  of the function $F$ is defined to be $$d_{F}=\sup_{t\geq 0}{tF^{'}(t)\over F(t)}\quad (
 l_{F}=\inf_{t\geq 0}{tF^{'}(t)\over F(t)})$$
 \end{defn}
We assume that $d_{F}$ is finite. The main result in this section
is the following theorem which give  monotonicity formula for the
functional $\Phi_{F}$.
\begin{thm}
 \ Let $(M,g)$ be a complete Riemannian manifold with a pole $x_{0}$. Assume that there exist two positive functions $h_{1}(r)$ and $h_{2}(r)$ such that
\begin{equation}h_{1}(r)\left[g-dr\otimes dr\right]\leq\hbox{Hess}(r)\leq h_{2}(r)\left[g-dr\otimes dr\right]
\end{equation}
on $M\setminus\{x_{0}\}$ where $r(x)=d_{g}(x,x_{0})$. Let $\phi :
\mathbb{R}^{+}\rightarrow \mathbb{R}^{+}$ be a $C^{2}$-function
such that $\phi(0)=0$, $\phi^{'}(t)>0$ on $]0,+\infty[$ and
$$\phi^{'}(r)h_{2}(r)\geq\phi^{''}(r)$$ on $M\setminus\{x_{0}\}$.
If $u : (M,g)\rightarrow(N,h)$ satisfies an
$\Phi_{F}$-conservation law, then for $0<R_{1}\leq R_{2}$ we have
\begin{equation}{1\over e^{G(R_{1})}}\int_{ B(x_{0},R_{1})}F\left({\Vert u^{*}h\Vert^{2}\over 4}\right)dv_{g}\leq{1\over e^{G(R_{2})}}\int_{B(x_{0},R_{2})}F\left({\Vert u^{*}h\Vert^{2}\over 4}\right)dv_{g}\end{equation}
where $G(R)$ is a  primitive of the lower bound of
$$ R\rightarrow\left(\phi^{'}(R)\right)^{-1}\displaystyle\inf_{ B(x_{0},R)}\left(\phi^{''}(r)+((m-1)h_{1}(r)-4d_{F}h_{2}(r))\phi^{'}(r)\right).$$
In particular, if
$$
\int_{ B(x_{0},R)}F\left({\Vert u^{*}h\Vert^{2}\over
4}\right)dv_{g} =o(e^{G(R)})$$ then $u$ is a constant map
\end{thm}
\begin{proof}
Denoted $D=B_{R}(x_{0})$ the geodesic ball of radius $R$ centred
at $x_{0}$. Taking $X=\phi^{'}(r){\partial\over\partial r}\in
T_{x_{0}}M$ where ${\partial\over\partial r}$ denoted unit radial
vector field. Choosing a local orthonormal frame field $\Big\{
e_{1},\cdots,e_{m-1},e_{m}={\partial\over\partial r}\Big\}$ on
$M$. Since $u$ satisfies an $\Phi_{F}$-conservation law, applying
formula (8) to $D=B(x_{0},R)$ and
$X=\phi^{'}(r){\partial\over\partial r}$  we have
\begin{eqnarray*}\int_{B_{R}(x_{0})}
<S_{\Phi_{F,u}}&,&\nabla\theta_{X}>dv_{g}=\int_{\partial B_{R}(x_{0})}S_{\Phi_{F,u}}(X,\nu)ds_{g}\qquad\quad{}\\
&=&\int_{\partial B_{R}(x_{0})}F({\Vert u^{*}h\Vert^{2}\over
4})g(X,\nu)ds_{g}-
\int_{\partial B_{R}(x_{0})}F^{'}({\Vert u^{*}h\Vert^{2}\over 4})h(\sigma_{u}(X),du(\nu))ds_{g}\\
&=&\phi^{'}(R)\left\{\int_{\partial B_{R}(x_{0})}F({\Vert
u^{*}h\Vert^{2}\over 4})ds_{g}
-\int_{\partial B_{R}(x_{0})}F^{'}({\Vert u^{*}h\Vert^{2}\over 4})h(\sigma_{u}({\partial\over\partial r}),du({\partial\over\partial r}))ds_{g}\right\}\\
&=&\phi^{'}(R)\left\{\int_{\partial B_{R}(x_{0})}F({\Vert
u^{*}h\Vert^{2}\over 4})ds_{g}
-\int_{\partial B_{R}(x_{0})}F^{'}({\Vert u^{*}h\Vert^{2}\over 4})\sum_{i=1}^{m}h(du(e_{i}),du({\partial\over\partial r}))^{2}ds_{g}\right\}\\
&\leq &\phi^{'}(R)\int_{\partial B_{R}(x_{0})}F({\Vert
u^{*}h\Vert^{2}\over 4})ds_{g}
\end{eqnarray*}
Now, we will compute the item $<S_{\Phi_{F,u}},\nabla\theta_{X}>$
on the left hand side. For this purpose, using local orthonormal
frame field $\left\{e_{1},\cdots,e_{m-1},{\partial\over\partial
r}\right\}$, it is easy to see that
$$\nabla_{\partial\over\partial
r}X=\phi^{''}{\partial\over\partial
r},\quad\nabla_{e_{i}}X=\phi^{'}\sum_{k=1}^{m}\hbox{Hess}(r)(e_{i},e_{k})e_{k},\quad
1\leq i\leq m-1, $$
$$\hbox{div}X=\phi^{''}+\phi^{'}\sum_{k=1}^{m-1}\hbox{Hess}(r)(e_{k},e_{k})$$
where Hess(.) denoted the Hessian operator, i.e
$$\hbox{Hess}(r)(e_{i},e_{j})=\nabla_{e_{i}}\nabla_{e_{j}}r-\left(\nabla_{e_{i}}e_{j}\right)r.$$
So
$$
h(\sigma_{u}(e_{\alpha}),du(e_{\beta}))g(\nabla_{e_{\alpha}}X,e_{\beta})
=\phi^{'}\hbox{Hess}(r)(e_{i},e_{j})h(\sigma_{u}(e_{i}),du(e_{j}))
+h(\sigma_{u}({\partial\over\partial
r}),du({\partial\over\partial r}))
$$
 Then
 \begin{eqnarray*}
 <S_{\Phi_{F,u}},\nabla\theta_{X}>&=&F\left({\Vert u^{*}h\Vert^{2}\over 4}\right)\hbox{div}X -
 F^{'}\left({\Vert u^{*}h\Vert^{2}\over 4}\right)\sum_{\alpha,\beta=1}^{m}h(\sigma_{u}(e_{\alpha}),du(e_{\beta}))g(\nabla_{e_{\alpha}}X,e_{\beta})\\
 &=&F\left({\Vert u^{*}h\Vert^{2}\over 4}\right)(\phi^{''}+\phi^{'}\sum_{k=1}^{m-1}\hbox{Hess}(r)(e_{k},e_{k}))-\\
 &{}& F^{'}\left({\Vert u^{*}h\Vert^{2}\over 4}\right)\Big\{\phi^{'}\sum_{i,j=1}^{m-1}\hbox{Hess}(r)(e_{i},e_{j})h(\sigma_{u}(e_{i}),du(e_{j}))
 +\phi^{''}h(\sigma_{u}({\partial\over\partial r}),du({\partial\over\partial r}))\Big\}\\
 &\geq & F\left({\Vert u^{*}h\Vert^{2}\over 4}\right)\left(\phi^{''}+(m-1)h_{1}\phi^{'}\right)-\\
 &{}& F^{'}\left({\Vert u^{*}h\Vert^{2}\over 4}\right)\Big\{\phi^{'}h_{2}\sum_{i=1}^{m-1}h(\sigma_{u}(e_{i}),du(e_{i}))+
 \phi^{''}h(\sigma_{u}({\partial\over\partial r},du({\partial\over\partial r})))\Big\}\\
 &=&F\left({\Vert u^{*}h\Vert^{2}\over 4}\right)\left(\phi^{''}+(m-1)h_{1}\phi^{'}\right)-\\
 &{}& \phi^{'}h_{2}F^{'}\left({\Vert u^{*}h\Vert^{2}\over 4}\right)\Vert u^{*}h\Vert^{2}+
 (\phi^{'}h_{2}-\phi^{''})\sum_{i=1}^{m}(h(du({\partial\over\partial r}),du(e_{i})))^{2}\\
 &\geq &\left(\phi^{''}+((m-1)h_{1}-4d_{F}h_{2})\phi^{'}\right)F\left({\Vert u^{*}h\Vert^{2}\over 4}\right)
 \end{eqnarray*}
 Hence
 $$<S_{\Phi_{F,u}},\nabla\theta_{X}>\geq\left(\phi^{''}(r)+((m-1)h_{1}(r)-4d_{F}h_{2}(r))\phi^{'}(r)\right)F\left({\Vert u^{*}h\Vert^{2}\over 4}\right)$$
 and
 $$
\int_{\partial B_{R}(x_{0})}F\left({\Vert u^{*}h\Vert^{2}\over
4}\right)ds_{g}\geq H(R)\int_{ B_{R}(x_{0})}F\left({\Vert
u^{*}h\Vert^{2}\over 4}\right)dv_{g}$$ where $H(R)$ is a lower
bound of $$
R\rightarrow\left(\phi^{'}(R)\right)^{-1}\displaystyle\inf_{
B(x_{0},R)}\left(\phi^{''}(r)+((m-1)h_{1}(r)-4d_{F}h_{2}(r))\phi^{'}(r)\right)$$
The coarea formula implies that $${d\over
dR}\int_{B_{R}(x_{0})}F\left({\Vert u^{*}h\Vert^{2}\over
4}\right)dv_{g}=\int_{\partial B_{R}(x_{0})}F\left({\Vert
u^{*}h\Vert^{2}\over 4}\right)ds_{g}$$ Hence
\begin{equation}
{\displaystyle{d\over
dR}\displaystyle\int_{B_{R}(x_{0})}F\left({\Vert
u^{*}h\Vert^{2}\over
4}\right)dv_{g}\over\displaystyle\int_{B_{R}(x_{0})}F\left({\Vert
u^{*}h\Vert^{2}\over 4}\right)dv_{g}}\geq H(R)\end{equation} for
almost every $R>0$. By integration (12) over $[R_{1},R_{2}]$, we
have
$$
\log\int_{B_{R_{2}}(x_{0})}F\left({\Vert u^{*}h\Vert^{2}\over
4}\right)dv_{g}-\log\int_{B_{R_{1}}(x_{0})}F\left({\Vert
u^{*}h\Vert^{2}\over 4}\right)dv_{g}\geq G(R_{2})-G(R_{1})$$ This
proves monotonicity inequality. The constancy of $u$ follows by
letting $R_{2}$ to infinity in (11).
\end{proof}
The rest of this section is devoted to derive some Liouville-type
results under some explicit curvatures conditions on $(M,g)$ with
a pole $x_{0}$. The radial curvature $K_{r}$ ( respectively radial
Ricci curvature $Ric_{r}$) of $M$  is the restriction of the
sectional curvature function ( respectively the Ricci curvature
function ) to all planes which contain the unit vector $\nabla r$
in $T_{x}M$ tangent to the unique geodesic joining $x_{0}$ to $x$
and pointing away from $x_{0}$. The tensor $g-dr\otimes dr$ is
trivial on the radial direction $\nabla r$ and equal to $g$ on the
orthogonal complement $[\nabla r]^{\perp}$. We regard how $K_{r}$
varies as long as we have Hessian comparison estimates with bounds
satisfying (10) with  $\phi(t)={t^{2}\over 2}$. We collect some
Liouville-type results for maps satisfying an
$\Phi_{F}$-conservation law in the following two theorems.
\begin{thm}
 Let $ u : (M,g)\rightarrow (N,h)$ be a $C^{2}$ map satisfying an $\Phi_{F}$-conservation law. The $u$ is constant  provided one of the following conditions is satisfied :\\
 (i) $-\alpha^{2}\leq K_{r}\leq -\beta^{2}$ with $\alpha>0,\beta>0$ and $(m-1)\beta-4d_{F}\alpha\geq 0  $ and $$\int_{B_{R}(x_{0})}F\left({\Vert u^{*}h\Vert^{2}\over 4}\right)dv_{g}=o\left(R^{(m-{4d_{F}\alpha\over\beta})}\right)$$
 (ii) $K_{r}=0$ with $m-4d_{F}>0$ and
 $$\int_{B_{R}(x_{0})}F\left({\Vert u^{*}h\Vert^{2}\over 4}\right)dv_{g}=o\left(R^{(m-4d_{F})}\right)$$
(iii) $-{A\over(1+r^{2})^{1+\epsilon}}\leq
K_{r}\leq-{B\over(1+r^{2})^{1+\epsilon}}$ with $\epsilon>0, A\geq
0, 0\leq B<2\epsilon$ and $1+(m-1)(1-{B\over
2\epsilon})-4d_{F}e^{A\over 2\epsilon}>0$ and
$$\int_{B_{R}(x_{0})}F\left({\Vert u^{*}h\Vert^{2}\over 4}\right)dv_{g}=o\big(R^{\big((m-1)(1-{B\over 2\epsilon})-4d_{F}e^{A\over 2\epsilon}+1\big)}\big)$$
(iv) $-\alpha^{2}\leq K_{r}\leq 0,\ Ric_{r}\leq -\beta^{2}$ where
$\alpha>0,\beta>0,\ \beta-4d_{F}\alpha\geq 0$ and
$$\int_{B_{R}(x_{0})}F\left({\Vert u^{*}h\Vert^{2}\over 4}\right)dv_{g}=o\big(R^{2(1-{2d_{F}\alpha\over\beta})}\big)$$
\end{thm}
\begin{thm}
 Let $ u : (M,g)\rightarrow (N,h)$ be a $C^{2}$ map satisfying an $\Phi_{F}$-conservation law. Suppose that $d_{F}\leq {1\over 4}$ and the radial curvature $K_{r}$ of $M$ satisfy :\\
$-Ar^{2q}\leq K_{r}\leq -Br^{2q}$ with $A\geq B>0, q>0$ and
$(m-1)B_{0}-4d_{F}\sqrt{A}\coth\sqrt{A}\geq 0$  where
$B_{0}=\min\left\{1,-{q+1\over 2}+\sqrt{B+{(q+1)^{2}\over
4}}\right\}$. Then for $1<R_{1}\leq R_{2}$ we have
$$
{1\over R_{1}^{\lambda}} \int_{B_{R_{1}}(x_{0})\setminus
B_{1}(x_{0})}F\left({\Vert u^{*}h\Vert^{2}\over
4}\right)dv_{g}\leq {1\over R_{2}^{\lambda}}
\int_{B_{R_{2}}(x_{0})\setminus B_{1}(x_{0})}F\left({\Vert
u^{*}h\Vert^{2}\over 4}\right)dv_{g} $$ where
$\lambda=1+(m-1)B_{0}-4d_{F}\sqrt{A}\coth\sqrt{A}$. In particular
if
$$\int_{B_{R}(x_{0})\setminus B_{1}(x_{0})}F\left({\Vert u^{*}h\Vert^{2}\over 4}\right)dv_{g}=o\left(R^{\lambda}\right)$$ then $u$ is constant on $M\setminus B_{1}(x_{0})$.
\end{thm}
For the proof, we will need the following Hessian comparison
theorems of Greene-Wu [6].
\begin{lem}
 Let $(M,g)$ be a complete Riemanian manifold with a pole $x_{0}$. Let $K_{r}$ and $Ric_{r}$ are the respectively the radial sectional and Ricci curvatures of $M$. Then \\
(a) If $-\alpha^{2}\leq K_{r}\leq -\beta^{2}$ with
$\alpha>0,\beta>0$ then
$$\beta\coth(\beta r)\left[g-dr\otimes dr\right]\leq\hbox{Hess}(r)\leq\alpha\coth(\alpha r)\left[g-dr\otimes dr\right]$$
(b) If $K_{r}=0$, then $${1\over r}\left[g-dr\otimes
dr\right]=\hbox{Hess}(r)$$ (c) If $-A(1+r^{2})^{-1-\epsilon}\leq
K_{r}\leq-B(1+r^{2})^{-1-\epsilon}$ with $\epsilon >0, A\geq 0$
and $0\leq B<2\epsilon$, then $${1-{B\over 2\epsilon}\over
r}\left[g-dr\otimes dr\right]\leq\hbox{Hess}(r)\leq{e^{A\over
2\epsilon}\over r}\left[g-dr\otimes dr\right]$$ (d) If
$-\alpha^{2}\leq K_{r}\leq 0$ and $Ric_{r}\leq -\beta^{2}$ where
$\alpha>0,\beta>0$ then
$$\Delta r\geq\beta\coth(\beta r) ,\quad \hbox{Hess}(r)\leq\alpha\coth(\alpha r)[g-dr\otimes dr]$$
(e) If $-Ar^{2q}\leq K_{r}\leq -Br^{2q}$ with $A\geq B>0$ and
$q>0$, then $$B_{0}r^{q}\left[g-dr\otimes
dr\right]\leq\hbox{Hess}(r)\leq(\sqrt{A}\coth\sqrt{A})r^{q}\left[g-dr\otimes
dr\right]$$ for $r\geq 1$, where $B_{0}=\min\left\{1,-{q+1\over
2}+\sqrt{B+{(q+1)^{2}\over 4}}\right\}.$
\end{lem}
\subsection{Proof of Theorem 3.3}
\begin{proof}\noindent In order to use theorem 3.2, we fix  $\phi(t)={t^{2}\over 2}$.\\ \\
\noindent{\it Case (i)}. Since $-\alpha^{2}\leq K_{r}\leq
-\beta^{2}$ with $\alpha>0,\beta>0$, by comparison theorem
$$\beta\coth(\beta r)\left[g-dr\otimes dr\right]\leq\hbox{Hess}(r)\leq\alpha\coth(\alpha r)\left[g-dr\otimes dr\right]$$
Since $(m-1)\beta-4d_{F}\alpha\geq 0$, by theorem
\begin{eqnarray*}
H(R)&=\left(\phi^{'}(R)\right)^{-1}\displaystyle\inf_{ B(x_{0},R)}\left(\phi^{''}(r)+((m-1)h_{1}(r)-4d_{F}h_{2}(r))\phi^{'}(r)\right)\\
&={1\over R}\displaystyle\inf_{ B(x_{0},R)}\left(1+((m-1)\beta\coth(\beta r)-4d_{F}\alpha\coth(\alpha r)  )r\right)\\
&={1\over R}\inf_{t\in [0,R]}(1+(m-1)\beta t\coth(\beta t)-4d_{F}\alpha t\coth(\alpha t))\\
&={1+(m-1)\beta-4d_{F}\alpha\over R}
\end{eqnarray*}
Thus $G(R)=\log R^{(1+(m-1)\beta-4d_{F}\alpha)}$ which implies the monotonicity inequality.\\ \\
\noindent{\it Case (ii)}. Since $K_{r}=0$ by comparison theorem
$${1\over r}\left[g-dr\otimes dr\right]=\hbox{Hess}(r)$$
In this case
\begin{eqnarray*}
H(R)&=\left(\phi^{'}(R)\right)^{-1}\displaystyle\inf_{ B(x_{0},R)}\left(\phi^{''}(r)+((m-1)h_{1}(r)-4d_{F}h_{2}(r))\phi^{'}(r)\right)\\
&={m-4d_{F}\over R}
\end{eqnarray*}
and $G(R)=\log R^{(m-4d_{F})}$ which implies the monotonicity inequality.\\ \\
\noindent{\it Case (iii)}. Since
$-{A\over(1+r^{2})^{1+\epsilon}}\leq
K_{r}\leq-{B\over(1+r^{2})^{1+\epsilon}}$ with $\epsilon>0, A\geq
0, 0\leq B<2\epsilon$ by comparison theorem
$${1-{B\over 2\epsilon}\over r}\left[g-dr\otimes dr\right]\leq\hbox{Hess}(r)\leq{e^{A\over 2\epsilon}\over r}\left[g-dr\otimes dr\right]$$
In this case
\begin{eqnarray*}
H(R)&=\left(\phi^{'}(R)\right)^{-1}\displaystyle\inf_{ B(x_{0},R)}\left(\phi^{''}(r)+((m-1)h_{1}(r)-4d_{F}h_{2}(r))\phi^{'}(r)\right)\\
&={1+(m-1)(1-{B\over\epsilon})-4d_{F}e^{A\over 2\epsilon}\over R}
\end{eqnarray*}
and $G(R)=\log R^{(1+(m-1)(1-{B\over\epsilon})-4d_{F}e^{A\over 2\epsilon})}$ which implies the monotonicity inequality.\\ \\
\noindent{\it Case (iv)}. Since $-\alpha^{2}\leq K_{r}\leq 0,\
Ric_{r}\leq -\beta^{2}$ where $\alpha\geq\beta>0$, by comparison
theorem
$$\Delta r\geq\beta\coth(\beta r),\quad\hbox{and}\quad \hbox{Hess}(r)\leq\alpha\coth(\alpha r)[g-dr\otimes dr]$$
If $\beta-4d_{F}\alpha\geq 0$, following the proof of theorem
\begin{eqnarray*}
{1\over\phi^{'}(R)}(\Delta\phi(r)& -4d_{F}h_{2}(r)\phi^{'}(r))={1\over R}(1+r\Delta r-4d_{F}\alpha r\coth(\alpha r))\\
&\geq{1+\beta r\coth(\beta r)-4d_{F}\alpha r\coth(\alpha r)\over R}\\
&\geq H(R)={2(1-{2d_{F}\alpha\over\beta})\over R}
\end{eqnarray*}
and $G(R)=\log R^{2(1-{2d_{F}\alpha\over\beta})}$ which implies the monotonicity inequality.\\
\end{proof}
\subsection{Proof of Theorem 3.4}
\begin{proof} By monotonicity formula (9)
$$\int_{\partial B(x_{0},R)}S_{\Phi_{F,u}}(X,\nu)ds_{g}-\int_{\partial B(x_{0},1)}S_{\Phi_{F,u}}(X,\nu)ds_{g}=\int_{B(x_{0},R)\setminus B(x_{0},1)}
\left<S_{\Phi_{F,u}},\nabla\theta_{X}\right>dv_{g}$$ Since
$-Ar^{2q}\leq K_{r}\leq -Br^{2q}$ with $A\geq B>0, q>0$, by
comparison theorem
$$B_{0}r^{q}\left[g-dr\otimes dr\right]\leq\hbox{Hess}(r)\leq(\sqrt{A}\coth\sqrt{A})r^{q}\left[g-dr\otimes dr\right]$$ for $r\geq 1$, where $B_{0}=\min\left\{1,-{q+1\over 2}+\sqrt{B+{(q+1)^{2}\over 4}}\right\}$.\\
Take $X=r{\partial\over\partial r}$, following the proof of
theorem ? we have
\begin{eqnarray*}\left<S_{\Phi_{F},u},\nabla\theta_{X}\right>&\geq(1+(m-1)B_{0}r^{q+1}-4d_{F}\sqrt{A}\coth\sqrt{A})
r^{q+1})F\left({\Vert u^{*}h\Vert^{2}\over 4}\right)\\
&\geq(1+(\lambda-1)r^{q+1})F\left({\Vert u^{*}h\Vert^{2}\over 4}\right)\\
&\geq\lambda F\left({\Vert u^{*}h\Vert^{2}\over 4}\right)
\end{eqnarray*}
and
\begin{eqnarray*}
S_{\Phi_{F},u}\left(X,{\partial\over\partial r}\right)&=F\left({\Vert u^{*}h\Vert^{2}\over 4}\right)-F^{'}\left({\Vert u^{*}h\Vert^{2}\over 4}\right)<\sigma_{u}({\partial\over\partial r}),du({\partial\over\partial r})>\quad\hbox{on}\quad\partial B_{1}(x_{0})\\
S_{\Phi_{F},u}\left(X,{\partial\over\partial
r}\right)&=RF\left({\Vert u^{*}h\Vert^{2}\over
4}\right)-RF^{'}\left({\Vert u^{*}h\Vert^{2}\over
4}\right)<\sigma_{u}({\partial\over\partial
r}),du({\partial\over\partial r})>\quad\hbox{on}\quad\partial
B_{R}(x_{0})
\end{eqnarray*}
Hence
\begin{eqnarray*}
R\int_{\partial B_{R}(x_{0})}\Big\{F\left({\Vert u^{*}h\Vert^{2}\over 4}\right)-F^{'}\left({\Vert u^{*}h\Vert^{2}\over 4}\right)&<\sigma_{u}({\partial\over\partial r}),du({\partial\over\partial r})>\Big\}ds_{g}-\\ &\int_{\partial B_{1}(x_{0})}\Big\{F\left({\Vert u^{*}h\Vert^{2}\over 4}\right)-F^{'}\left({\Vert u^{*}h\Vert^{2}\over 4}\right)<\sigma_{u}({\partial\over\partial r}),du({\partial\over\partial r})\Big\}ds_{g}\\
&\geq\lambda\int_{B_{R}(x_{0})\setminus B_{1}(x_{0})}F\left({\Vert
u^{*}h\Vert^{2}\over 4}\right)dv_{g}
\end{eqnarray*}
Since $\Vert u^{*}h\Vert^{2}\geq<\sigma_{u}({\partial\over\partial
r}),du({\partial\over\partial r})>$ and $tF^{'}(t)\leq d_{F}F(t)$
\begin{equation*}
\int_{\partial B_{1}(x_{0})}\Big\{F\left({\Vert
u^{*}h\Vert^{2}\over 4}\right)-F^{'}\left({\Vert
u^{*}h\Vert^{2}\over 4}\right)<\sigma_{u}({\partial\over\partial
r}),du({\partial\over\partial
r})\Big\}ds_{g}\geq(1-4d_{F})\int_{\partial
B_{1}(x_{0})}F\left({\Vert u^{*}h\Vert^{2}\over
4}\right)ds_{g}\geq 0
\end{equation*}
Hence for $R>1$
\begin{equation*}
R\int_{\partial B_{R}(x_{0})}\Big\{F\left({\Vert
u^{*}h\Vert^{2}\over 4}\right)-F^{'}\left({\Vert
u^{*}h\Vert^{2}\over 4}\right)<\sigma_{u}({\partial\over\partial
r}),du({\partial\over\partial
r})>\Big\}ds_{g}\geq\lambda\int_{B_{R}(x_{0})\setminus
B_{1}(x_{0})}F\left({\Vert u^{*}h\Vert^{2}\over 4}\right)dv_{g}
\end{equation*}
Coarea formula then implies
\begin{equation*}
{{d\over dR}\int_{B_{R}(x_{0})\setminus B_{1}(x_{0})}F\left({\Vert
u^{*}h\Vert^{2}\over 4}\right)dv_{g}\over
\int_{B_{R}(x_{0})\setminus B_{1}(x_{0})}F\left({\Vert
u^{*}h\Vert^{2}\over 4}\right)dv_{g}}\geq{\lambda\over R}
\end{equation*}
for almost all $R\geq 1$. Integrating over $[R_{1},R_{2}]$, we get
\begin{eqnarray*}
\log\left(\int_{B_{R_{2}}(x_{0})\setminus
B_{1}(x_{0})}F\left({\Vert u^{*}h\Vert^{2}\over
4}\right)dv_{g}\right)-&
\log\left(\int_{B_{R_{1}}(x_{0})\setminus B_{1}(x_{0})}F\left({\Vert u^{*}h\Vert^{2}\over 4}\right)dv_{g}\right)\\
&\geq\lambda\log R_{2}-\lambda\log R_{1}
\end{eqnarray*}
which implies monotonicity inequality.
\end{proof}
\section{Constant Dirichlet boundary-value problems}
 In this section we deal with constant Dirichlet boundary-value problems for maps satisfying an $\Phi_{F}$-conservation law. As
 in [5], we introduce  starlike domains with $C^{1}$-boundaries which generalize $C^{1}$-convex domains.
\begin{defn} A bounded domain $D\subset (M,g)$ with $C^{1}$-boundary $\partial D$ is called starlike if there exist an interior point $x_{0}\in D$ such that
$$
<{\partial\over\partial r_{x_{0}}},\nu>_{g}\Big|_{\partial D}\geq
0
$$
where $\nu$ is the unit normal to $\partial D$ and ${\partial\over\partial r_{x_{0}}}$ is
the unit vector field such that for any $x\in D\setminus\{x_{0}\}\cup\partial D$, ${\partial\over\partial r_{x_{0}}}(x)$
is the unit vector tangent to the unique geodesic joining $x_{0}$ to $x$ and pointing away from $x_{0}$.
\end{defn}
\begin{thm}
Let $ u : (M,g)\rightarrow (N,h)$ be a $C^{2}$ map and $D\subset M$ be a starlike domain. Assume that $l_{F}\geq{1\over 2}$ and $u\big|_{\partial D}$ is constant. If $u$ satisfies  an $\Phi_{F}$-conservation law, then $u$ is constant on $D$ provided one of the following conditions is satisfied :\\
(i) $-\alpha^{2}\leq K_{r}\leq -\beta^{2}$ with $\alpha>0,\beta>0$ and $(m-1)\beta-4d_{F}\alpha\geq 0$,\\
(ii) $K_{r}=0$ with $m-4d_{F}>0$,\\
(iii) $-{A\over(1+r^{2})^{1+\epsilon}}\leq K_{r}\leq-{B\over(1+r^{2})^{1+\epsilon}}$ with $\epsilon>0, A\geq 0, 0\leq B<2\epsilon$ and $1+(m-1)(1-{B\over 2\epsilon})-4d_{F}e^{A\over 2\epsilon}>0$,\\
(iv) $-\alpha^{2}\leq K_{r}\leq 0,\ Ric_{r}\leq -\beta^{2}$ where
$\alpha>0,\beta>0,\ \beta-4d_{F}\alpha\geq 0$.
\end{thm}
\begin{proof}
Let the vector field $X=r_{x_{0}}\nabla r_{x_{0}}$. Under the
radial curvatures conditions, following the proof of theorem 2.2,
localized on $\overline D$, we get
$$\left<S_{\Phi_{F}},\nabla\theta_{X}\right>\geq KF\left({\Vert u^{*}h\Vert^{2}\over 4}\right)\quad \hbox{on}\quad D$$
where
\begin{eqnarray*}
K&=1+(m-1)\beta-4d_{F}\alpha &\hbox{case (i)}\\
&=m-4d_{F}&\hbox{case (ii)}\\
&=1+(m-1)(1-{B\over\epsilon})-4d_{F}e^{A\over 2\epsilon}&\hbox{case (iii)}\\
&=2(1-{2d_{F}\alpha\over\beta})&\hbox{case (vi)}
\end{eqnarray*}
Let $x\in\partial D$ and choose  a local orthonormal frame field
$\{ e_{1},\cdots,e_{m-1},\nu\}$ on $T_{x}M$ such that $\{
e_{1},\cdots,e_{m-1}\}$ is a orthonormal frame field on
$T_{x}\partial D$. Since $u|_{\partial D}$ is constant, we get
$du(e_{i})=0,\ i=1,\cdots m-1$ and $du({\partial\over\partial
r_{x_{0}}})=<{\partial\over\partial r_{x_{0}}},\nu>_{g}du(\nu)$.
Hence at $x$
\begin{eqnarray*}
S_{\Phi_{F},u}(X,\nu)&=r_{x_{0}}S_{\Phi_{F},u}({\partial\over\partial r},\nu)\\
&=r_{x_{0}}\left(F\left({\Vert u^{*}h\Vert^{2}\over 4}\right)<{\partial\over\partial r},\nu>_{g}-F^{'}\left({\Vert u^{*}h\Vert^{2}\over 4}\right)<\sigma_{u}({\partial\over\partial r}),du(\nu)>_{h}\right)\\
&=r_{x_{0}}<{\partial\over\partial r_{x_{0}}},\nu>_{g}\left(F\left({\Vert u^{*}h\Vert^{2}\over 4}\right)-F^{'}\left({\Vert u^{*}h\Vert^{2}\over 4}\right)\Vert u^{*}h\Vert^{2}\right)\\
&\leq r_{x_{0}}<{\partial\over\partial
r_{x_{0}}},\nu>_{g}(1-4l_{F})F\left({\Vert u^{*}h\Vert^{2}\over
4}\right)\leq 0
\end{eqnarray*}
Hence $S_{\Phi_{F},u}(X,\nu)\leq 0$ on $\partial D$. Since $u$
satisfies an $\Phi_{F}$-conservation law, by monotonicity formula
$$
0\leq K\int_{D}F\left({\Vert u^{*}h\Vert^{2}\over
4}\right)dv_{g}\leq\int_{D}\left<S_{\Phi_{F}},\nabla\theta_{X}\right>dv_{g}=\int_{\partial
D}S_{\Phi_{F},u}(X,\nu)ds_{g}\leq 0
$$
which implies that $u\big|_{D}$ is constant.
\end{proof}
\section{Generalized Chern type results}
In this section, we deal with the following  mean $F$-curvature
type equation for   $TN$-valued sections over  on $(M,g)$ :
$$
 u : (M,g)\rightarrow (N,h)\quad \hbox{ with}\quad
 \hbox{div}_{g}\sigma_{F,u}=s
$$
where $s : M\rightarrow u^{-1}TN$ is a $C^{1}$ section,
$$\sigma_{F,u}=F^{'}\left({\Vert u^{*}h\Vert^{2}\over
4}\right)\sigma_{u}$$ where $\sigma_{u}$ is the $\mathbb R$-valued
$1$-form  on $M$  defined by
$$\sigma_{u}=\sum_{i=1}^{m}<du(.),du(e_{i})>_{h}du(e_{i})$$ We
observe that $$\Vert\sigma_{F,u}\Vert\leq F^{'}\left({\Vert
u^{*}h\Vert^{2}\over 4}\right)\Vert u^{*}h\Vert^{2}\leq
4d_{F}F\left({\Vert u^{*}h\Vert^{2}\over 4}\right)$$ where $d_{F}$
is the $F$-degree. \noindent We recall that $(M,g)$ is said to
have the doubling property if there exist a constant $D(M)>0$ such
that $\forall R>0,\ \forall x\in M$,
$$
\hbox{Vol}_{g}(B(x,2R))\leq D(M)\hbox{Vol}_{g}(B(x,R)).$$
\begin{thm} Let $(M,g)$ be a complete non-compact Riemannian manifold with doubling property. Let  $u :(M,g)\rightarrow (N,h)$ be a $C^{2}$ map such that
$$\hbox{div}_{g}\sigma_{F,u}=s$$
 off a bounded set $K\subset M$ where $s$ is a parallel $C^{1}$ section of $u^{-1}TN$ over $M$ i.e ${\tilde\nabla}s=0$. Let $\phi : ]0,+\infty[\rightarrow ]0,+\infty[$ be an increasing function . Assume that
$$\sup_{x\in B(R)}\Vert s(x)\Vert\phi(r(x))=o\left(R\right)$$
 If    $$\limsup_{R\rightarrow\infty}{1\over\hbox{Vol}_{g}(B(R))}\int_{B(R)}{\Vert\sigma_{F,u}\Vert\over\phi(r(x))}dv_{g}
 <\infty$$
then $\inf_{x\in M\setminus K}\Vert s(x)\Vert=0$.
\end{thm}
\begin{proof}
For any section $t$ of $u^{-1}TN$ over $M$, let $Z$ be the vector
field on $M$ defined by $$<Z,X>_{g}=<t,\sigma_{F,u}(X)>_{h}$$ for
all vector fields $X$ on $M$. We choose a local orthonormal frame
field $\{e_{i}\}_{i=1}^{m}$ on $M$ and we compute
\begin{eqnarray*}\hbox{div}_{g}Z &=&\sum_{i=1}^{m}<\nabla^{M}_{e_{i}}Z,e_{i}>_{g}\\
&=&\sum_{i=1}^{m}e_{i}<Z,e_{i}>_{g}
+\sum_{i=1}^{m}\big\{<{\tilde\nabla}_{e_{i}}t,i_{e_{i}}\sigma_{F,u}>
+<t,{\tilde\nabla}_{e_{i}}\sigma_{F,u}(e_{i})>\big\}\\
&=&<\sum_{i=1}^{m}\theta_{e_{i}}\wedge{\tilde\nabla}_{e_{i}}t,\sigma_{F,u}>+
<t,\sum_{i=1}^{m}{\tilde\nabla}_{e_{i}}\sigma_{F,u}(e_{i})>\\
&=&<{\tilde\nabla}t,\sigma_{F,u}>+<t,\hbox{div}_{g}\sigma_{F,u}>
\end{eqnarray*}
where $\tilde\nabla$ is the induced connection on  $u^{-1}TN$ from
$\nabla^{M}$ and $\nabla^{N}$ and $\theta_{e_{i}}$ the dual
$1$-form of $e_{i}$. Then for any bounded open set $D\subset
M\setminus K$ with smooth boundary $\partial D$, we have
\begin{equation}
\int_{\partial
D}\left<t,\sigma_{F,u}(\nu)\right>ds_{g}=\int_{D}\left<{\tilde\nabla}
t,\sigma_{F,u}\right>_{h}dv_{g}+\int_{D}\left<t,\hbox{div}_{g}\sigma_{F,u}\right>_dv_{g}
\end{equation}
where $\nu$ denotes the unit outward normal vector field on
$\partial D$. The formula (5.1) with $t=\psi s$, where $\psi\in
C^{2}(M,\mathbb R^{+})$, gives
\begin{eqnarray*}
\int_{D}\psi\Vert s(x)\Vert_{h}^{2}dv_{g}&=&-\int_{D}<d\psi\otimes
s,\sigma_{F,u}>dv_{g}-\int_{D}\psi<{\tilde\nabla}s,\sigma_{F,u}>dv_{g}\\
&{}&\quad +\int_{\partial
D}\psi<s,\sigma_{F,u}(\nu)>ds_{g}
\end{eqnarray*}
Since $K\subset M$ is compact, choose a sufficiently large $R_{0}
<R$ such that $K\subset B(x_{0},R_{0})$. Let $0\leq\psi\leq 1$ be
the cut-off function i.e $\psi=1$ on $\overline B(x_{0},R),\
\psi=0$ off $B(x_{0},2R)$, and $\vert\nabla\psi\vert\leq{C\over
R}$. The formula (14) with $D=B(x_{0},2R)\setminus B(x_{0},R_{0})$
implies
\begin{eqnarray*}
\inf_{x\in M\setminus K}\Vert s(x)\Vert^{2}\left(1-{V(R))\over
V(R_{0}}\right)&\leq &
{C\over RV(R)}\int_{B(x_{0},2R)\setminus B(x_{0},R_{0})}\Vert s\Vert\Vert\sigma_{F,u}\Vert dv_{g}+\\
&{}& {1\over V(R)}\int_{B(x_{0},2R)\setminus
B(x_{0},R_{0})}\Vert{\tilde\nabla}s\Vert\Vert\sigma_{F,u}\Vert
dv_{g}+ \\
&{}& {1\over V(R)}\int_{\partial B(x_{0},R_{0})}\Vert
s\Vert\Vert\sigma_{F,u}(\nu)\Vert ds_{g}
\end{eqnarray*}
where $V(R)=\hbox{Vol}_{g}(B(x_{0},R))$. Since $\sup_{x\in
B(R)}\Vert{\tilde\nabla}s(x)\Vert=o\left({1\over\phi(R)}\right)$,
for each $\eta>0$ 
$$\sup_{x\in
B(R)}\Vert{\tilde\nabla}s(x)\Vert\leq{\eta\over\phi(R)}\quad R\geq
R_{1}$$ By mean value inequality $$\Vert
s(x)-s(x_{0})\Vert\leq\eta{R\over\phi(R)}$$ in $B(x_{0},R)$ for
$R\geq R_{1}$. By doubling property, the inequality (15) become
\begin{eqnarray*}
\inf_{x\in M\setminus K}\Vert s(x)\Vert^{2}\left(1-{V(R_{0})\over
V(R)}\right)&\leq & CD(M)\left(\eta +\Vert
s(x_{0})\Vert{\phi(2R)\over R}\right){1\over
V(2R)}\int_{B(2R)}{\Vert\sigma_{F,u}\Vert\over\phi(r(x))}dv_{g}+\\
&{}& 
\eta{
 D(M)\over V(2R)}\int_{B(2R)}{\Vert\sigma_{F,u}\Vert\over\phi(r(x))}dv_{g}+
{1\over V(R)}\int_{\partial B(x_{0},R_{0})}\Vert
s\Vert\Vert\sigma_{F,u}(\nu)\Vert ds_{g}
\end{eqnarray*}
where $D(M)$ is the constant in doubling property. Since $(M,g)$
is complete, the doubling property implies that $(M,g)$ has
infinite volume [9]. Letting  $R$  to infinity,  we get
$$
\inf_{x\in M\setminus K}\Vert s(x)\Vert^{2}\leq \eta
D(M)(C+1)\limsup_{R\rightarrow\infty}{1\over
V(2R)}\int_{B(2R)}{\Vert\sigma_{F,u}\Vert\over\phi(r(x))}dv_{g}
$$
and $\inf_{x\in M\setminus K}\Vert s(x)\Vert=0$ since  $\eta$ is
arbitrary.
\end{proof}
\noindent The following is an analogous of  Chern's result [4] :
\begin{cor} Let $(M,g)$ be a complete non-compact Riemannian manifold with doubling property and $(N,h)$ be a Riemannian manifold. Let  $u :(M,g)\rightarrow (N,h)$ be a $C^{2}$ map such that
$\hbox{div}_{g}\sigma_{F,u}=s$
  where $s$ is a constant section of $u^{-1}TN$. Let $\phi : ]0,+\infty[\rightarrow ]0,+\infty[$ be an increasing function such that $\phi(R)=o(R)$. If
\begin{equation*}\limsup_{R\rightarrow\infty}{1\over V(R)}\int_{B(R)}{\Vert\sigma_{F,u}\Vert\over\phi(r(x))}dv_{g}
<\infty\end{equation*} then $u$ is a stationary map for the
functional $\Phi_{F}$.
\end{cor}

\end{document}